\documentclass[twoside]{amsart}
\usepackage{latexsym}
\usepackage{amssymb,amsmath,amsopn}
\usepackage[dvips]{graphicx}   
\usepackage{color,epsfig}      

               {\begin{list}{}{\leftmargin#1\rightmargin#2}\item{}}%
               {\end{list}}
               
\newtheorem{thm}{Theorem}
\newtheorem{lem}{Lemma}
\theoremstyle{definition}
\newtheorem{defn}{Definition} 
\newtheorem{rem}{Remark}

\newtheorem{conj}{Conjecture}
\newtheorem{prob}{Problem}
\newtheorem*{example}{Example}
\newtheorem{cor}{Corollary}
\renewcommand{\Re}{\mathbb R}

\newcommand{\B}{\mathbf B}

\renewcommand{\P}{\mathcal{P}}

\newcommand{\dif}{\;\mathrm{d}}
\def\bea{\begin{eqnarray}}
\def\eea{\end{eqnarray}}
\newcommand{\F}{\mathcal{F}}
\newcommand{\K}{\mathcal{K}}
\renewcommand{\O}{\mathcal{O}}
\newcommand{\rin}{\rho^{in}}
\newcommand{\rex}{\rho^{ex}}

\DeclareMathOperator{\perim}{perim}
\DeclareMathOperator{\relint}{relint}
\DeclareMathOperator{\relbd}{relbd}
\DeclareMathOperator{\area}{area}
\DeclareMathOperator{\inter}{int}
\DeclareMathOperator{\bd}{bd}

\DeclareMathOperator{\conv}{conv}

\DeclareMathOperator{\vol}{vol}

\DeclareMathOperator{\surf}{surf}

\begin{document}
\title[Robust equilibria]{The robustness of equilibria on convex solids}
\author[G. Domokos, Z. L\'angi]{G\'abor Domokos, Zsolt L\'angi}

\address{G\'abor Domokos, Dept. of Mechanics, Materials and Structures, Budapest University of Technology,
M\H uegyetem rakpart 1-3., Budapest, Hungary, 1111}
\email{domokos@iit.bme.hu}
\address{Zsolt L\'angi, Dept. of Geometry, Budapest University of Technology,
Egry J\'ozsef u. 1., Budapest, Hungary, 1111}
\email{zlangi@math.bme.hu}


\begin{abstract}
We examine the minimal magnitude of perturbations necessary to change the number $N$ of static equilibrium points
of a convex solid $K$. We call the normalized volume of the minimally necessary truncation \emph{robustness} and we seek
shapes with maximal robustness for fixed values of $N$. While the \emph{upward}
robustness (referring to the increase of $N$) of smooth, homogeneous convex solids is known to be zero, little is known about their \emph{downward}
robustness. The difficulty of the latter problem is related to the coupling (via integrals) between
the geometry of the hull $\bd K$ and the location of the center of gravity $G$.
Here we first investigate two simpler,  decoupled problems
by examining truncations of $\bd K$ with $G$ fixed, and 
displacements of $G$ with $\bd K$ fixed, leading to the concept of \em external \rm
and \em internal \rm robustness, respectively. In dimension 2, we find that for any fixed number $N=2S$,
the convex solids with both maximal external and
maximal internal robustness are regular $S$-gons. Based on this result we conjecture that regular polygons
have maximal downward robustness also in the original, coupled problem.
We also show that in the decoupled problems, $3$-dimensional regular polyhedra have maximal internal robustness, however,  only under additional constraints.
Finally, we prove results for the full problem in case of 3 dimensional solids. These results appear to explain why monostatic
pebbles (with either one stable, or one unstable point of equilibrium) are found so rarely in Nature.
\end{abstract}
\maketitle

\section{Introduction}

Ever since the work of Archimedes \cite{Archimedes}, the study of equilibrium points of a convex solid, 
with respect to its center of gravity $G$ has been a fundamental question of statics.
The number $N$ of equilibria is characteristic of the shape, it has been applied to classify turtle shells \cite{DV} and beach pebbles \cite{DSSV}.
These classifications are based on the number and type of equilibrium points of a convex, homogeneous solid; 
in 2 dimensions we have $S$ stable and $S=U$ unstable points, and it was shown \cite{DomokosRuina} that any
\emph{ equilibrium class} $\{S\}$ 
is non-empty if $S>1$. In 3 dimensions we have, in addition to the previous two types of equilibria, $H$ saddle points; 
based on the Poincar\'e-Hopf Theorem \cite{Arnold} we have $S+U-H=2$. It is known \cite{Varkonyi} that equilibrium classes $\{S,U\}$ are non-empty
for any $S,U>0$.
It is a natural question to ask how difficult it is to change the equilibrium class of a convex solid;
we call this property \emph{robustness}, and our aim is to introduce possible approaches to this concept.
The notion of robustness is physically motivated by erosion processes \cite{Bloore}, where small
amounts of material are being abraded by collisions \cite{DSV},\cite{SDWH},\cite{Krapivsky} and friction \cite{DomokosGibbons}.

The full mathematical problem can be defined by asking for the minimal normalized volume of a truncation of a convex, homogeneous solid $K$, under which $N$ will change;
we call this scaled volume the \emph{robustness} of $K$ and denote it by $\rho(K)$. (To avoid confusion, we sometimes will refer to this quantity as the \emph{full robustness}.)
Depending on the sign of the change, we may define \emph{upward} and \emph{downward} robustness.
A recent result \cite{DLS2} shows that the upward robustness of generic, smooth convex bodies
(i.e. convex Morse functions on the sphere, cf. \cite{Milnor}) is zero.
This phenomenon is related to the existence of additional numbers of equilibria on finely discretized curves and surfaces \cite{DLS}  
and is also distantly related to Zamfirescu's result \cite{Zamfirescu} on the existence of infinitely many equilibria on the boundary of a typical (and thus neither smooth nor polyhedral) convex body.
Nevertheless, the downward robustness (i.e. the relative volume of the smallest trunction of a smooth convex body decreasing $N$)
is apparently not zero. We call the supremum of the downward robustness in any equilibrium class the \emph{robustness
of the equilibrium class} and denote it by $\rho_S$ and $\rho_{S,U}$ in 2 and 3 dimensions, respectively. Finding
the robustness of an arbitrary equilibrium class appears to be a nontrivial problem. Beyond being mathematically challenging, it also looks
rather interesting from the point of view of natural abrasion processes. Our goal in this paper is to deliver some partial results
which may serve as a basis for the intuition about the solution of the full problem.

The main difficulty of the latter lies in the nontrivial coupling via integrals between the hull $\bd K$ and the center of gravity $G$. In the
first part of this paper we solve two simpler, decoupled problems. In the first case
we seek \emph{internal robustness}, defined as the minimal (normalized) distance necessary to move $G$ at fixed $\bd K$ leading to a change in $N$.
The concept of internal robustness is phyiscally motivated by material inhomogeneities. From
the mathematical point of view, this problem is closely related to the geometry of caustics \cite{Poston},
and has attracted recent interest in the context of inhomogeneous polyhedra (cf. \cite{DawsonFinbow} and \cite{Heppes}). Here we do not distinguish
between upward/downward robustness, rather we compute their minimum.
The other decoupled problem leads to the concept of \emph{external robustness}, defined as the minimal (scaled) truncation of $K$ at fixed $G$, leading to a change in $N$.
Similarly to the full problem, (in fact, as a consequence of the same theorem \cite{DLS2}) here again the upward robustness is zero.
We show that in the plane, both decoupled problems lead to the same result: for any fixed $N=2S$, the convex shapes with maximal internal \emph{or} external (downward) robustness
are regular $S$-gons, thus we determine the internal and external robustness
for all planar equilibrium classes. This result suggests that regular polygons may have maximal robustness in the original, full problem as well. In the case of internal robustness,
we show also that platonic solids have maximal robustness, however, only under an additional constraint. 
After exploring the decoupled problems we investigate the full (downward) robustness of some selected equilibrium classes in 3 dimensions,
in particular, we show that if $S,U <3$ then $\rho_{S,U}=1$. We also explore the \emph{partial robustness} of convex 3D solids,
measuring the difficulty to either reduce $S$ or $U$. Our results offer one plausible  explanation to the geological puzzle, why
monostatic pebbles (with either $S=1$ or $U=1$) are found so extremely rarely in Nature.

After introducing basic notions and notations in Section~\ref{sec:notations},
we investigate external and internal robustness in the plane in Sections~\ref{sec:external}
and \ref{sec:internal}, respectively. Subsequently, we discuss internal robustness
in the $3$-dimensional space in Section~\ref{sec:internal_3D} and explore
the robustness of some selected equilibrium classes in Section \ref{sec:full}.
Finally, in Section~\ref{sec:summary} we make additional remarks about smoothness and structural stability and
use the previous results to formulate conjectures about the full (downward) robustness
of equilibrium classes in 3 dimensions.

\section{Basic notations}\label{sec:notations}

In this paper, we deal with convex bodies in the Euclidean spaces $\Re^2$ and $\Re^3$, of dimension $2$ or $3$, respectively, where, by a convex body, we mean a compact convex set with nonempty interior.
For a point $p$, we let $|p|$ be the Euclidean norm of $p$, and denote the Euclidean unit ball of
the ambient space by $\B$.

We distinguish three subclasses of convex bodies: by $\P_n$, $\O_n$ and $\K_n$, we denote the families of
$n$-dimensional convex polytopes, convex bodies with smooth ($C^\infty$-class) boundary, and convex bodies with piecewise smooth boundary, respectively.

Let $K \in \K_2$ be a convex body, and $p \in \inter K$.
We say that $q \in \bd K$ is an \emph{equilibrium point} of $K$ with
respect to $p$, if the line passing through $q$ and perpendicular to $q-p$,
supports $K$.
Clearly, if $\bd K$ is smooth at $q$, then this condition is equivalent to saying that $q$ is a critical point
of the Euclidean distance function $z \mapsto |z-p|$, $z \in \bd K$.

We call the equilibrium at $q$ \emph{nondegenerate}, if one of the following holds:
\begin{itemize}
\item if $\bd K$ is smooth at $q$, then the second derivative of $z \mapsto |z-p|$, $z \in \bd K$ at $q$ is not zero,
\item if $\bd K$ is not smooth at $q$, then both angles between $p-q$ and one of the two one-sided
tangent half lines of $\bd K$ at $q$ are acute.
\end{itemize}
Note that if $K \in \O_2$ or $K \in \P_2$, then this definition reduces to the usual
concept of nondegeneracy in these classes.
In the case of a smooth point, we call the nondegenerate equilibrium point $q$ \emph{stable} or \emph{unstable}, if
the second derivative at $q$ is positive or negative, respectively.
In the nonsmooth case, we call the equilibrium point \emph{unstable}.
The fact that the numbers of the stable and unstable equilibrium points of any $K \in \K_2$ are equal
if $K$ has only nondegenerate equilibrium points, follows from the Poincar\'e-Hopf Theorem (cf. \cite{Arnold}).
These two types of points form an alternating sequence in $\bd K$.

These definitions can be naturally adapted to convex bodies in $\O_3$ using the Euclidean distance function,
(distinguishing three types of nondegenerate equilibrium points: unstable, saddle and stable points; depending on the number of negative eigenvalues of the Hessian) and also to convex polytopes in $\Re^3$. We note that in the latter case,
unstable, saddle and stable points are vertices, relative interior points of edges and of faces, respectively.
For the definition of nondegeneracy in the piecewise smooth case in higher dimensions, the reader is referred to \cite{SOA}.
If $K$ has only nondegenerate equilibrium points, then the Poincar\'e-Hopf Theorem yields that
\[
S - H + U = 2,
\]
where $S$, $H$ and $U$ denote the numbers of the stable, saddle and unstable points of the body.

Throughout the paper, we deal with bodies that have only nondegenerate equilibrium points.

For simplicity, we denote the family of plane convex bodies $K \in \K_2$, with $S$ stable points
with respect to their centers of gravity, by $\{S \}$.
For convex bodies in $\O_3$ or in $\P_3$, we may define the class $\{S,U\}$ similarly,
where $S$ and $U$ denote the numbers of the stable and the unstable points with respect to the center of
gravity of the body.

If $K$ has $S$ stable points with respect to some $p \in \inter K$, let
$\F_{<}(K,p)$ be the family of (convex) subsets of $K$ such that any $K'
\in \F_{<}(K,p)$ has strictly less than $S$ stable points with respect
to $p$.
We obtain a similar notion, which we denote by $\F_{<}(K)$, if the
reference point is not fixed, but in the cases of both $K$ and $K'$ it
is the center of gravity of the corresponding body.
We define the (full) robustness of a planar convex body $K \in \K_2$ as follows.

\begin{defn}\label{defn:full}
Let $K \in \{ S \}$. Then we define the
\emph{downward robustness} (or simply \emph{robustness}) of $K$ as the
quantity \[ \rho(K) = \frac{\min \{ \area(K \setminus K') : K' \in
\F_{<}(K) \}}{\area(K)}. \]
\end{defn} 
\noindent In $\Re^3$, we may define $\rho(K)$ in an analogous way.
For brevity, we set $\rho_S = \sup \{
\rho(K): K \in \{ S\} \}$ and $\rho_{S,U} = \sup \{ \rho(K): K \in \{
S,U\} \}$.

In Section~\ref{sec:external}, our aim is to investigate robustness with
the reference point fixed, which we define below.
Note that in this definition the reference point need not be the center
of gravity of the body.

\begin{defn}\label{defn:external} Let $K \in \K_2$ and $p \in \inter K$.
Assume that $K$ has $S$ stable points with respect to $p$. We define the
\emph{downward external robustness of $K$} (or simply \emph{external
robustness}) with respect to $p$ as the quantity \[ \rex(K,p) =
\frac{\min \{ \area(K \setminus K') : K' \in \F_{<}(K,p) \}}{\area(K)}.
\] \end{defn}

\noindent For simplicity, we set
\[
\rex_S = \sup \{ \rex(K,G) : K \in \{S \}\},
\]
where $G$ denotes the center of gravity of $K$.
In $\Re^3$, we may define $\rex(K,p)$ and $\rex_{S,U}$ in an analogous way.

In Sections~\ref{sec:internal} and \ref{sec:internal_3D}, we examine \emph{internal robustness},
which we define below.
\begin{defn}\label{defn:internal}
Let $K \in \K_2$ and $p \in \inter K$. Assume that $K$ has $S$ stable points with respect to $p$.
Let $R(K,p) \subseteq \Re^2$ denote the set of the points such that $K$ has $S$ stable
points with respect to any point of $R(K,p)$.
The \emph{internal robustness} of $K$ with respect to $p$ is
\[
\rin(K,p) = \frac{\min \left\{
|q-p| : q \notin R(K,p) \right\}}{\perim K}  ,
\]
where $\perim K$ is the perimeter of $K$.
\end{defn}

\begin{rem}
By compactness, it is easy to see that
if $K$ has $S$ stable points with respect to $p$,
then $p$ has a neighborhood $U$ such that with respect to any $z \in U$, $K$ has $S$ stable points with respect to $z$.
Thus, if $K$ has only nondegenerate equilibrium points with respect to $p$,
then $\rin(K,p) > 0$.
\end{rem}

\begin{rem}
Definition \ref{defn:internal} does not distinguish between upward and downward robustness.
If the latter are understood as the minimal distance necessary to move $p$ to achieve
increase/decrease of $N$ then our definition refers to their minimum.
\end{rem}
\noindent Similarly like for full and external robustness, we set
\[
\rin_S = \sup \{ \rin(K,G) : K \in \{S \}\},
\]
wher $G$ is the center of gravity of $K$. In $\Re^3$, we define $\rin(K,p)$ and $\rin_{S,U}$ similarly, by replacing $\perim K$ by the square root of the surface area $\surf K$ of the body.

We note that it is easy to show that for each $S$ and $U$ the infimum of $\rho(K,p)$, $\rex(K,p)$ and $\rin(K,p)$ is zero
in any class $\{S\}$ and $\{S,U\}$.

\section{External robustness}\label{sec:external}

Our main result in this section is the following.

\begin{thm}\label{thm:external}
Let $K \in \K_2$ contain the origin in its interior, and assume that $K$ has $S \geq 3$ stable points
with respect to $o$.
Then
\[
\rex(K,o)\leq \frac{\tan\frac{\pi}{S} - \frac{\pi}{S}}{S\tan\frac{\pi}{S}},
\]
with equality if, and only if $K$ is a regular $S$-gon and $o$ is its center.
\end{thm}

In light of the notations introduced in the previous section, we may reformulate this statement
in the following way, for the special case that $o$ is the center of gravity of $K$.

\begin{cor}
For any $S \geq 3$, we have
\[
\rex_S = \frac{\tan\frac{\pi}{S} - \frac{\pi}{S}}{S\tan\frac{\pi}{S}},
\]
and the planar convex bodies in $\{ S \}$ with maximal external robustness are the regular $S$-gons.
\end{cor}
First, we prove Lemma~\ref{lem:angle}.

\begin{lem}\label{lem:angle}
Let $Q$ be a convex quadrangle with vertices $a,b,c,d$ in counterclockwise order. Let the angles of $Q$
at $b$ and $d$ be right angles, and let $\alpha \in (0,\pi)$ denote the angle at $a$.
Set $r = \max \{ |b-a|, |d-a| \}$, and for any plane convex body $C \subseteq Q$ containing $[a,b] \cup [a,d]$,
let $X(Q,C) = \{ z \in C: |z-a| \geq r \}$.
Then, the quantity $\mu(Q,C)=\frac{\area(X(Q,C))}{\area(C)}$ is maximal over the convex quadrilaterals $Q$, with given angles, and over plane convex bodies $C \subseteq Q$ containing $[a,b] \cup [a,d]$ if, and only if, $C=Q$ and $Q$ is
symmetric about $[a,c]$.
\end{lem}

\begin{proof}
It is easy to see that (if we permit $Q$ to be a degenerate quadrangle) the maximum of $\mu(Q,C)$ is attained
for some $Q$ and $C$, hence, we show that if $|b-a| \neq |d-a|$ or $C \neq Q$, then $\mu(Q,C)$ is not maximal.

Without loss of generality, let $r=|b-a| \geq |d-a|$, and assume that $X_C \neq \emptyset$.
Let $z$ be the point of $C$ with $|z-a|=r$ and farthest from $b$, and let $H$ denote the closed half plane containing
$[d,z]$ on its boundary, and $a$ in its interior.

If $[d,z] \notin \bd C$, then we may remove a part of $C \setminus X_C$, and thus increase $\mu$.
If $X_C \neq (X_Q \cap H)$ then we may increase $\area(X(Q,C))$, and thus, increase $\mu$.
Consequently, it suffices to examine the case that $[d,z] \subset \bd C$ and $X(Q,C) = X(Q,Q) \cap H$.

Consider the case that $C \neq Q$, which yields that $z \notin [c,d]$.
Let $L$ be the line passing through $z$ and parallel to $[c,d]$, and $H'$ be
the closed half plane bounded by $L$ and containing $b$.
Let $c'$ and $d'$ denote the intersection point of $L$ with $[b,c]$ and $[d,a] \cup [a,b]$, respectively.
Set $Q'=\conv \{ a,b,c',d'\}$ if $d' \in [d,a]$, and $Q'=\conv \{ d',b,c'\}$ if $d' \in [a,b]$ (cf. Figure~\ref{fig:robex}).
Furthermore, set $C'=Q' \cap H'$. Observe that $\area(X(Q',C')) > \area(X(Q,C))$ and
$\area(C' \setminus X(Q',C')) < \area(C \setminus X_C)$. Thus, $\mu(Q,C) < \mu(Q',C')$,
which yields that if $\mu(Q,C)$ is maximal, then $C = Q$.

\begin{figure}[here]
\includegraphics[width=0.7\textwidth]{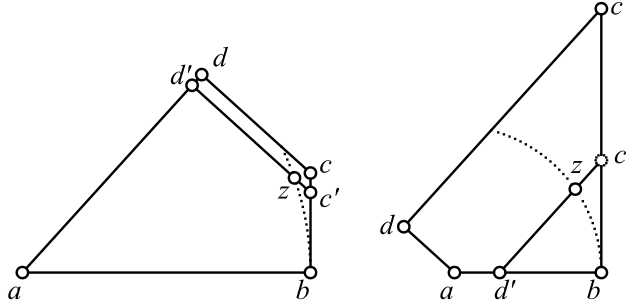}
\caption[]{An illustration for the proof of Lemma~\ref{lem:angle}}
\label{fig:robex}
\end{figure}

Finally, we examine the case that $C=Q$ and $|b-a| > |d-a|$.
Let $[a',d']$ be defined by the properties
\begin{itemize}
\item $a' \in [a,b]$ and $d' \in [c,d]$,
\item $[a',d']$ and $[a,d]$ are parallel, and
\item $|d'-a'| = |b - a'|$.
\end{itemize}
Let $Q' = C' = \conv \{ a', b,c,d' \}$. Observe that $X(Q,C) \subsetneq X(Q',C')$ and $Q' \subsetneq Q$.
Hence, $\mu(Q,C)$ is not maximal, and the assertion readily follows.
 \end{proof}

\begin{proof}[Proof of Theorem~\ref{thm:external}]
Let the stable points of $K$ be $s_1, s_2, \ldots, s_S$ in counterclockwise order in $\bd K$.
Let $K_i$ be the closed set of points $z$ of $K$ such that the position vectors $s_i$, $z$ and $s_{i+1}$
are in counterclockwise order around $o$.
For any $i=1,2,\ldots,n$, let $r_i = \max \{ |s_i|,|s_{i+1}| \}$
and set $X_i = \{ z \in K_i : |z| > r_i \}$.

We observe that if $K' \subset K$ has less than $S$ stable points with respect to $o$,
then for some value of $i$, $K' \cap X_i = \emptyset$.
Indeed, any such $K'$ has at most $S-1$ stable points.
On the other hand, if there is some point $z_i \in K' \cap X_i$ for every value of $i$, then
the sequence $|z_1|-|s_1|,|s_2|-|z_1|,|z_2|-|s_2|, \ldots, |s_1|-|z_S|$ alternates, 
which yields that the Euclidean distance function, defined on the points of $\bd C$,
has at least $S$ local minima; a contradiction.
Thus, noting that $o \notin X_i$ for any value of $i$, we obtain that
\[
\rex(K,o)= \frac{ \min \{ \area(X_i) : i=1,2,\ldots, S \}}{\area(K)}.
\]
We intend to maximize this quantity over $K \in \K_2$.

First, observe that if $\area(X_i)$ is not minimal for some value of $i$, then we may truncate $X_i$
and thus decrease $\area(K)$. Hence, we may assume that $\area(X_i)$ is the same quantity for every value of $i$.
Now, let $\alpha_i = \angle(s_i,o,s_{i+1})$.

For each $i$ with $\alpha_i < \pi$, let $K'_i = \conv \{ x_i,o,y_i,w_i\}$ be the right kite such that
\begin{itemize}
\item $x_i$ and $y_i$ are on the half lines, starting at $o$, that contain $s_i$ and $s_{i+1}$, respectively,
\item $\area(X'_i) = \area(X_i)$, where $X'_i = \{ z \in K'_i : |z| \geq  |s_i| \}$. 
\end{itemize}
Note that by Lemma~\ref{lem:angle}, we have $\area(K_i) \geq \area(K'_i)$ for every $i$, with equality
if, and only if $K_i = K'_i$.

\emph{Case 1}, if $\alpha_i < \pi$ for every value of $i$.
We let $K' = \bigcup_{i=1}^S K'_i$, $x = \area(X'_i)$, and note that $K'$ is not necessarily convex, and that $\area(X'_i)$ is independent of $i$.

Observe that $\rex(K,o)\leq \frac{x}{\area(K')}$, with equality if, and only if $K=K'$.
Thus, it suffices to show that $\frac{x}{\area(K')}$ is maximal if, and only if $K'$ is a regular $S$-gon, or in other words, if $\alpha_i = \frac{2\pi}{S}$ for every $i$. Since the maximum of $\frac{x}{\area(K')}$ is clearly attained,
we show only that if $\alpha_i \neq \alpha_{i+1}$ for some $i$, then this quantity is not maximal.

Assume that $\alpha_i \neq \alpha_{i+1}$ for some $i$.
Set $r_i = |x_i|$ and $r_{i+1} = |x_{i+1}|$.
Then
\[
x = r_i^2 \left( \tan\frac{\alpha_i}{2} - \frac{\alpha_i}{2} \right) = r_{i+1}^2 \left( \tan\frac{\alpha_{i+1}}{2} - \frac{\alpha_{i+1}}{2} \right) ,
\]
and
\[
\area(K_i)+\area(K_{i+1}) = r_i^2 \tan\frac{\alpha_i}{2} + r_{i+1}^2 \tan\frac{\alpha_{i+1}}{2} .
\]
Using the notation $\alpha = \frac{\alpha_i}{2}$ and $\beta=\frac{\alpha_{i+1}}{2}$,
the first condition can be transformed into the form
\[
\left( \frac{r_{i+1}}{r_i} \right)^2 = \frac{\tan\alpha - \alpha}{\tan\beta - \beta} ,
\]
which yields that
\[
A= \frac{\area(K_i)+\area(K_{i+1})}{x} = \frac{\frac{\tan\alpha - \alpha}{\tan\beta - \beta} \tan \beta + \tan \alpha}{\tan\alpha-\alpha} = \frac{\tan\alpha}{\tan\alpha - \alpha}+\frac{\tan\beta}{\tan\beta-\beta}.
\]

We need only show that under the constraint that $\alpha + \beta$ is constant, $A$ is minimal if and only if
$\alpha= \beta$. But this indeed holds, since for the function $f(\alpha) = \frac{\tan\alpha}{\tan\alpha - \alpha}$,
we have
\[
f''(\alpha)=\frac{2 \tan\alpha(\alpha^2-\sin^2\alpha)}{\cos^2\alpha(\tan\alpha-\alpha)^3} > 0
\]
for every $\alpha \in \left( 0, \frac{\pi}{2} \right)$, which yields that $f$ is strictly convex.

\emph{Case 2}, if $\alpha_i \geq \pi$ for some value of $i$.
Observe that in this case $\alpha_i \geq \pi$ for exactly one value of $i$.
Let this value be $S$, and set $\sum_{i=1}^{S-1} \alpha_i = \omega \leq \pi$.
Let $K' = K_S \cup \left( \bigcup_{i=1}^{S-1} K'_i \right)$, and $x = \area(X'_i)$ for some $i \neq S$.
Then $\rex(K,o) \leq \frac{x}{\area(K')}$.
Using the argument of Case 1, we have that if $\frac{x}{\area(K')}$ is maximal,
then $K'_1, K'_2, \ldots, K'_{S-1}$ are congruent right kites, with their angles at $o$
equal to $\frac{\omega}{n-1}$. According to our consideration, we have $\area(X_S) = \area(X'_i)$
and $\area(K_S) \geq \area(K'_i)$ for any $i \neq S$.
Thus, in this case we have
\[
\rex(K,o) \leq \frac{\tan\frac{\omega}{2(S-1)} - \frac{\omega}{2(S-1)}}{S\tan\frac{\omega}{2(S-1)}} <
\frac{\tan\frac{\pi}{S} - \frac{\pi}{S}}{S\tan\frac{\pi}{S}} =\rex(P,o),
\]
where $P$ is a regular $S$-gon, with the origin as its center.
\end{proof}

From the proof of Theorem~\ref{thm:external}, one can easily deduce Corollary~\ref{cor:monotonicity}.

\begin{cor}\label{cor:monotonicity}
Assume that $K \in \K_2$ has $S$ stable points with respect to $p \in \inter K$.
Then there exists $\varepsilon=\varepsilon(K)$ such that if $K' \subset K$ has less than $S$ stable points
with respect to $p \in \inter K'$, and $\area(K \setminus K') - \rex(K,p) \area(K) \leq \varepsilon$,
then $\rex(K',p) \geq \rex(K,p)$.
\end{cor}

\section{Internal robustness in the plane}\label{sec:internal}

Our main theorem is the following:

\begin{thm}\label{thm:internal}
For any $K \in \K_2$ and $p \in \inter K$, if $K$ has $S \geq 3$ stable points with respect to $p$, then
$\rin(K,p) \leq \frac{1}{2S}$, with equality if, and only if,
$K$ is a regular $S$-gon, and $p$ is its center.
\end{thm}

Similarly like in Section~\ref{sec:external}, we may reformulate Theorem~\ref{thm:internal} in terms of $\rin_S$.

\begin{cor}
For any $S \geq 3$, we have
\[
\rin_S=\frac{1}{2S},
\]
and the plane convex bodies $K \in \{ S \}$ with maximal internal robustness with respect to their centers of gravity
are the regular $S$-gons.
\end{cor}

In the proof we use Definition~\ref{defn:stripcover} and Lemmas~\ref{lem:discrete} and \ref{lem:inscribed}.

\begin{defn}\label{defn:stripcover}
Let $\F$ be a family of closed segments in $\Re^2$, and $A \subset \Re^2$ a set.
If for every $[a,b] \in F$ there is a closed infinite strip, containing $A$ and bounded by a pair of parallel lines $L_a$ and $L_b$ such that $a \in L_a$ and $b \in L_b$, then we say that \emph{$A$ admits a strip cover by the elements of $\F$}.
\end{defn}

\begin{lem}\label{lem:discrete}
Let $P \subset \Re^2$ be an $S$-gon of unit perimeter. If $B=q + \rho \B$ admits a strip cover by the sides of $P$,
then $\rho \leq \frac{1}{2S}$ with equality if, and only if, $P$ is a regular $S$-gon, and $q$ is its center.
\end{lem}

\begin{lem}\label{lem:inscribed}
Let $K \in \K_2$ have $S \geq 3$ stable points with respect to $p \in \inter K$.
Let $P$ denote the convex hull of the unstable points of $K$.
Then $B=p+\rin(K,p) \B$ admits a strip cover by the sides of $P$.
\end{lem}

First, we prove Theorem~\ref{thm:internal}, and then the two lemmas.

\begin{proof}[Proof of Theorem~\ref{thm:internal}]
Let $K \in \K_2$ be of unit perimeter and let $p \in \inter K$.
Assume that $K$ has $S \geq 3$ stable points with respect to $p$.
Let $P$ denote the convex hull of the unstable points of $K$.
Since $P \subset K$, $\perim P \leq \perim K = 1$, with equality if, and only if, $P=K$.
Thus, applying Lemmas~\ref{lem:discrete} and \ref{lem:inscribed}, we immediately have $\rin(K,p) \leq \frac{1}{2S}$.
On the other hand, if $\rin(K,p)=\frac{1}{2S}$, then $\perim P = 1$, and by Lemma~\ref{lem:discrete},
$P$ is a regular $S$-gon, and $p$ is its center. From this, it readily follows that $P=K$ is a regular $S$-gon,
and $p$ is its center.
\end{proof}

\begin{proof}[Proof of Lemma~\ref{lem:discrete}]
Observe that, for any side $[a,b]$ of $P$, the width of any infinite strip, bounded by the lines $L_a$ and $L_b$ 
with $a \in L_a$ and $b \in L_b$, is at most $|b-a|$, and here we have equality if, and only if $L_a$ and $L_b$
are perpendicular to $[a,b]$.
Thus, $2\rho$ is not greater than the length of a shortest side of $P$.
Since $\perim P =1$, it readily implies the inequality $\rho \leq \frac{1}{2S}$.

Assume that $\rho = \frac{1}{2S}$. Then $P$ is an equilateral $S$-gon, and for any side $[a,b]$ of $P$,
the strip $\conv (L_a, L_b)$, containing $B$, is perpendicular to $[a,b]$ and is circumscribed about $B$.
Thus, the center $q$ of $B$ is on the bisector of the side.
Since this holds for the bisector of each side of $P$, it follows that $q$ is the center of the circle
circumscribed about $P$. Note that if each vertex of an equilateral polygon lies on the same circle,
then the polygon is regular, which immediately implies the assertion.
\end{proof}

\begin{proof}[Proof of Lemma~\ref{lem:inscribed}]
Let $K \in \K_2$ be of unit diameter with $S \geq 2$ stable points with respect to the origin $o$.
Recall that the \emph{caustic} of a smooth curve (also called \emph{evolute}, cf. \cite{DoCarmo})
is the locus of the centers of curvature of the curve.
It is well-known (cf. \cite{Poston}), that a plane convex body with smooth boundary has degenerate equilibria only with respect to a point of its \emph{caustic}, and that otherwise the number $2S$ of the equilibria of $K$ changes
(and in this case it changes by two) if, and only if, the reference point transversally crosses the caustic (in $3$-dimensional space one of the two caustics). We use this observation in the following, more general form:
$\rin(K,o)$ is the largest number $\rho > 0$ such that
\begin{enumerate}
\item $\rho \inter \B$ contains no center of curvature at any smooth point of $\bd K$,
\item $\rho \inter \B$ contains no point of any one-sided inner normal half line at any nonsmooth point of $\bd K$.
\end{enumerate}

We parametrize $\bd K$ as $t \mapsto \underline{r}(t)$, where $t \in [0,1]$.
For $i=1,2,\ldots,S$, we denote the stable and unstable points of $K$ by $u_i=\underline{r}(t_i')$ and $s_i=\underline{r}(t_i)$,
labelling them in such a way that $0 < t_1 < t_1' < \ldots < t_n < t_n' < 1$.
Let $P= \conv \{ u_1, u_2, \ldots, u_S\}$.

Clearly, for every $i$,
\begin{enumerate}
\item the tangent lines of $\bd K$ at $u_i=\underline{r}(t_i)$ and at $s_i=\underline{r}(t_i')$ are perpendicular
to the position vectors $u_i$ and $s_i$, respectively if $u_i$ is a smooth point of $\bd K$
(note that $s_i$ is a smooth point by definition), and
\item the angle between $u_i$ and any of the two one-sided tangent lines
is acute, if $u_i$ is not a smooth point. 
\end{enumerate}
Thus, the angles of the triangle $\conv \{ o, ,u_i, u_{i+1}\}$ at $u_i$ and $u_{i+1}$
are acute, from which it follows that each side of $P$ contains a stable point with respect to $o$.

Observe also that as $\bd K$ is piecewise smooth and $K$ has only nondegenerate equilibria,
$\langle \underline{r}(t), \dot{\underline{r}}(t) \rangle < 0$ at any smooth point with $t \in (t_i,t_i')$
(that is, $|\underline{r}(t)|$ strictly increases), and it is positive at any smooth point with
$t \in (t_i',t_{i+1})$ ($|\underline{r}(t)|$ decreases).
This observation holds also for nonsmooth points, if we replace $\dot{\underline{r}}(t)$ by any of the two one-sided
derivatives of $\underline{r}(t)$.

Let $\alpha_i = \angle(u_i,o,s_i)$, $\beta_i = \angle(s_i,o,u_{i+1})$ and $l_i = |u_{i+1} - u_i|$.
We distinguish three cases.

\emph{Case 1}, $\alpha_i < \frac{\pi}{2}$ and $\beta_i < \frac{\pi}{2}$.

For simplicity, we imagine the segment $[u_i,u_{i+1}]$ as `horizontal', and assume that $o$ is `below'
this segment (cf. Figure~\ref{fig:rob1}).
Let $L$ and $L'$ denote the two lines, parallel to $[o,s_i]$, that pass through the points
$u_i$ and $u_{i+1}$, respectively.
Let $B$ and $B'$ be the closed disks, with $o$ as their centers, that touch $L$ and $L'$, respectively.
We show that both disks contain a center of curvature, or intersect a one-sided inner normal half line
at a nonsmooth point of the arc $\underline{r}([t_i,t_{i+1}])$.
This clearly implies that $\rin(K,o) \B$ is contained 
in a strip between two parallel lines passing through $u_i$ and $u_{i+1}$, respectively.

\begin{figure}[here]
\includegraphics[width=0.5\textwidth]{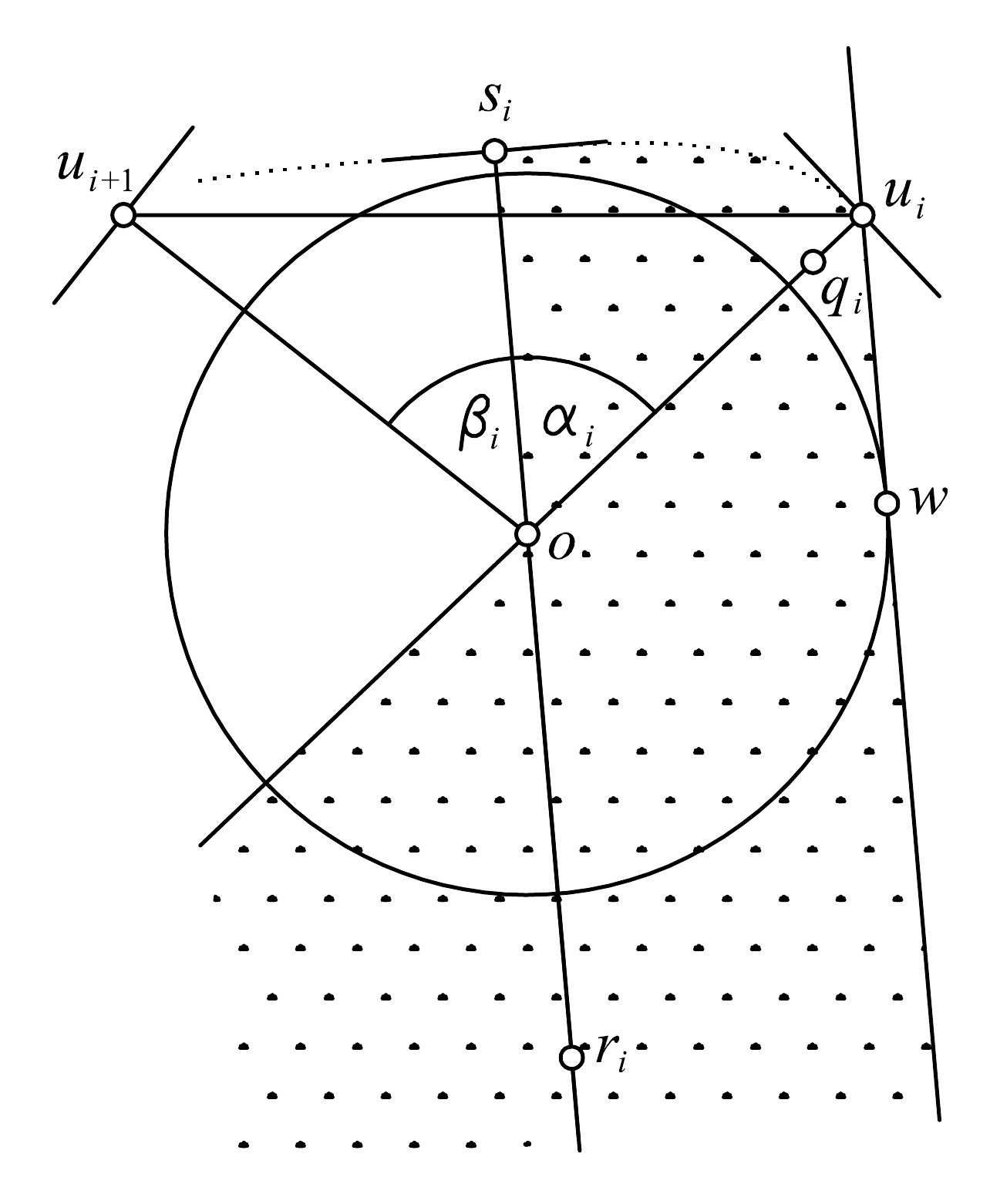}
\caption[]{The region in Case 1 of the proof of Lemma~\ref{lem:inscribed}}
\label{fig:rob1}
\end{figure}

We show only that $B$ contains a point of the caustic or a one-sided inner normal half line,
as for $B'$ we may repeat the same argument.

First, note that if $u_i$ is not a smooth point of $\bd K$, then the right-sided inner normal half line at $u_i$
intersects $B$. Thus, we may assume that $u_i$ is a smooth point of $\bd K$.

Let $w$ denote the tangent point of $B$ on $L$.
Observe that since $\alpha_i < \frac{\pi}{2}$, $u_i$ is upwards from $w$.
Let $D$ denote the region, containing $r_i$, that is bounded by the union of the following arcs:
\begin{itemize}
\item the closed half line in $L$, emanating from $u_i$ and containing $w$;
\item the points $\underline{r}(t)$ with $t \in [t_i,t_i']$;
\item the segment $[o,s_i]$;
\item the closed half ray in the line containing $[o,u_i]$, starting at $o$ and \emph{not} containing $u_i$.
\end{itemize}
This region is shown as a dotted domain in Figure~\ref{fig:rob1}.

Let $R_t$ denote the inner normal half line of $g_i=\underline{r}([t_i,t'_i])$ at $\underline{r}(t)$, if it exists.
Since $\langle \underline{r}(t), \dot{\underline{r}}(t) \rangle < 0$ for every $t \in (t_i,t_i')$,
we have that in this interval
$R_t \cap [o,s_i] = \emptyset$. Thus, $R_t \subset D$ for every $t \in [t_i,t_i']$.
Thus, the assertion readily follows if $g_i$ contains a nonsmooth point.

Now, consider the case that $g_i$ is a smooth curve, and let
$q_i$ and $r_i$ denote the centers of curvature at $u_i$ and $s_i$, respectively.
Observe that $q_i \in \relint [u_i,o]$ and $o \in \relint [s_i,r_i]$.
Furthermore, we have $q_i \in B$, or $r_i \in B$, or that $q_i$ and $r_i$ are not in the same
connected component of $D \setminus B$.
Since in the first two cases the assertion readily follows,
we may assume that $q_i$ and $r_i$ are not in the same connected component.

If the curvature of $\bd C$ is not zero at any point of $g_i$, then,
as in this case the caustic of $g_i$ is a continuous curve in $D$ that connects $q_i$ and $r_i$,
$B$ contains a center of curvature of $g_i$.
Consider the case that the curvature of  $g_i$ is zero at some point.
Let $\bar{t} \in [t_i,t_i']$ be the smallest value such that the curvature of $g_i$ is zero at $\underline{r}(\bar{t})$.
Then, the caustic of $\underline{r}([t_i,\bar{t}))$ is a continuous curve in $D$ that connects $u_i$ to a point
in the unbounded component of $D \setminus B$, and thus, contains a point in $B$.
This finishes the proof in Case 1.

\emph{Case 2}, $\alpha_i + \beta_i < \pi$, but one of the two angles is at least $\frac{\pi}{2}$.
Let, say, $\alpha_i \geq \frac{\pi}{2}$ (cf. Figure~\ref{fig:rob2}), which readily implies that $\beta_i < \frac{\pi}{2}$. Let $d$ denote the distance of $o$ and the line through $u_{i+1}$ and parallel to $[o,s_i]$.
This distance is the length of the segment $[u_{i+1},y]$, where $y$ is the orthogonal projection of $u_{i+1}$
on the line containing $[o,s_i]$. Let $x$ be the intersection point of $[u_i,u_{i+1}]$ and $[o,s_i]$.

\begin{figure}[here]
\includegraphics[width=0.5\textwidth]{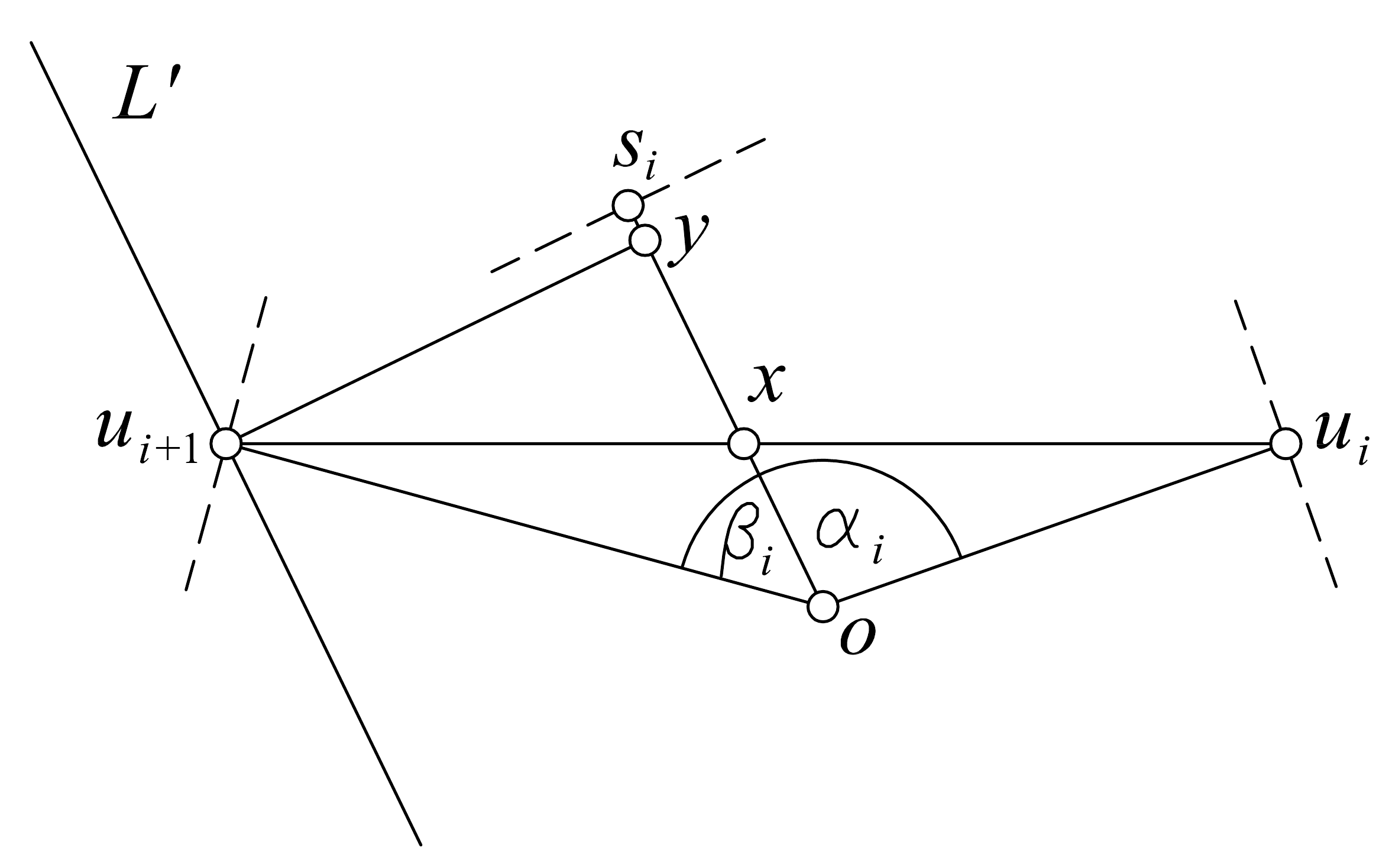}
\caption[]{An illustration for Case 2 of the proof of Lemma~\ref{lem:inscribed}}
\label{fig:rob2}
\end{figure}

Observe that by the argument in Case 1, we have that $d \geq \rin(K,o)$.
On the other hand, $\rin(K,o) \leq |u_i|$.
Indeed, if $u_i$ is a smooth point of $\bd K$, then $|u_i| \B$ contains the center of curvature at $u_i$,
and if $u_i$ is not a smooth point, then it intersects the right-hand side inner normal half line at $u_i$.

Consider the case that $|u_{i+1}-x| \leq |u_i-x|$ (cf. Figure~\ref{fig:rob2}).
Then, since $\rin(K,o) \leq d$, choosing $L$ and $L'$ parallel to $[o,s_i]$ and satisfying
$u_i \in L$, $u_{i+1} \in L'$, we have that $\rin(K,o) \B \subset \conv (L \cup L')$.

Now we examine the case that $|u_{i+1}-x| > |u_i-x|$.
Let $L$ and $L'$ be the lines perpendicular to $u_i$ such that $u_i \in L$ and $u_{i+1} \in L'$.
We show that $\rin(K,o) \B \subset \conv (L \cup L')$.
Let $z$ be the intersection point of $[u_i,u_{i+1}]$ with the line containing $o$ and parallel to $L$.
Observe that since $\alpha_i > \frac{\pi}{2}$, we have $z \in [x,u_i]$.
Thus, $|u_{i+1}-z| > |u_i-z|$, which yields that $|u_i| \B \subset \conv(L \cup L')$.
Since $|u_i| \geq \rin(K,o)$, it readily implies that $\rin(K,o) \subset \conv(L \cup L')$.

\emph{Case 3}, $\alpha_i + \beta_i \geq \pi$. From this condition, (using the terminology of Case 1)
it immediately follows that $o$ is `not below' the segment $[u_i,u_{i+1}]$, and thus, $o \notin \inter P$
(cf. Figure~\ref{fig:rob3}). Observe that the condition of Case 3 may hold for at most one value of $i$.
According to the previous cases, we have that for every $j \neq i$, $\rin(K,o) \B$ is contained 
in a strip bounded by two parallel lines passing through $u_j$ and $u_{j+1}$, respectively.
We may even observe that $p$ is contained in the halves of these strips, bounded by the corresponding sides of $P$,
that overlap $P$.

\begin{figure}[here]
\includegraphics[width=0.7\textwidth]{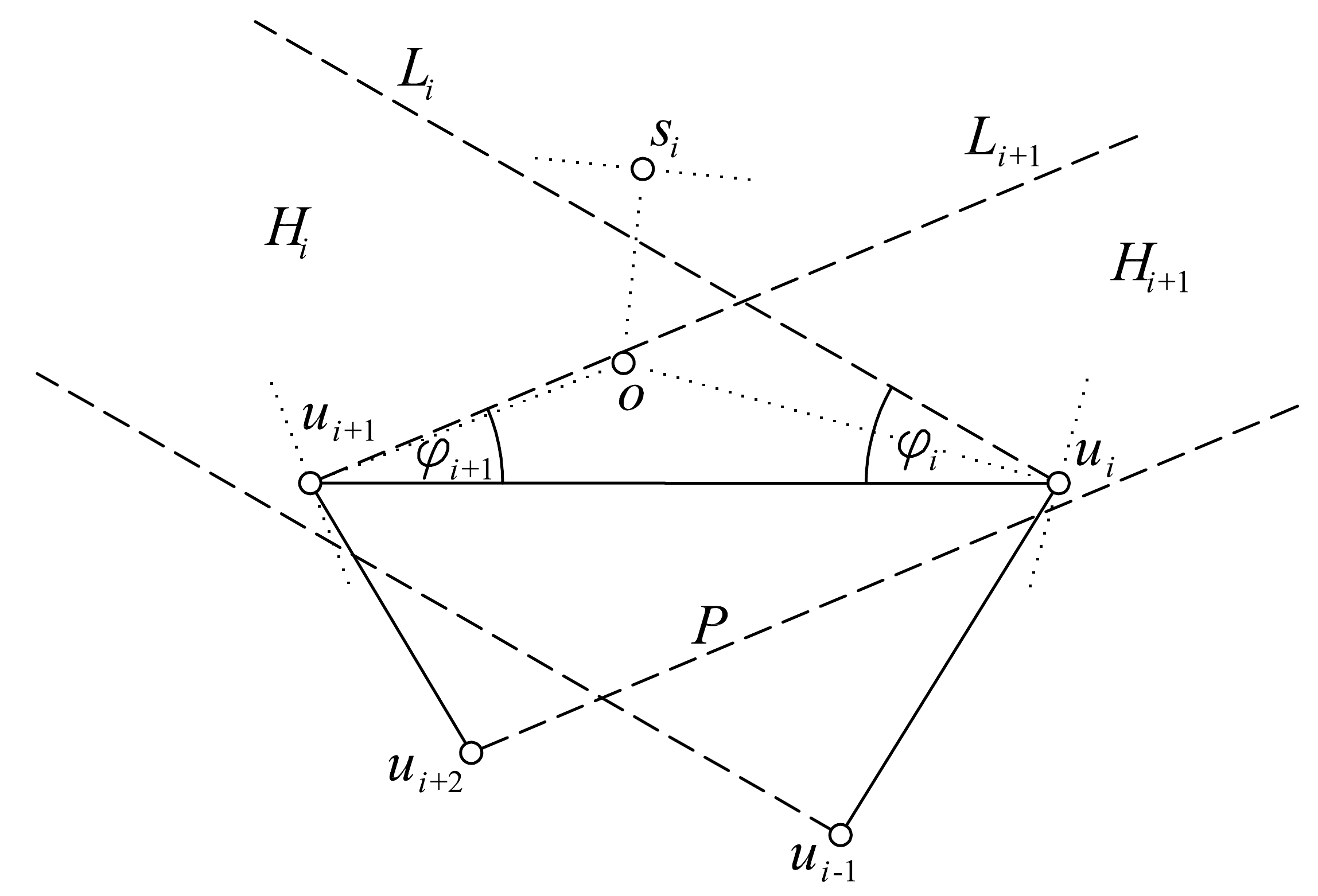}
\caption[]{An illustration for Case 3 of the proof of Lemma~\ref{lem:inscribed}}
\label{fig:rob3}
\end{figure}

We show that there is such a strip for $[u_i,u_{i+1}]$ as well.
Indeed, let $H_{i-1}$ denote the half of the infinite strip, belonging to $[u_{i-1},u_i]$,
that is bounded by $[u_{i-1},u_i]$ and contains $o$.
Define $H_{i+1}$ similarly for $[u_{i+1},u_{i+2}]$.
Let $L_i$ denote the ray in $\bd H_i$ starting at $u_i$, and $L_{i+1}$ denote the ray in $\bd H_{i+1}$
starting at $u_{i+1}$.

If $u_i \notin H_{i+1}$, then the strip, bounded by $L_{i+1}$ and its translate starting at $u_i$,
contains $H_{i+1}$, and thus also $\rin(K,o) \B^2$.
If $u_{i+1} \notin H_i$, we may apply a similar argument.
Thus, we may assume that $[u_i,u_{i+1}] \subset H_i \cap H_{i+1}$ (cf. Figure~\ref{fig:rob3}).
Let $\phi_i$ and $\phi_{i+1}$ denote the angle between $L_i$ and $[u_i,u_{i+1}]$, and the angle between $L_{i+1}$
and $[u_i,u_{i+1}]$, respectively.
Note that since $u_{i+1} \in H_i$, we have $\phi_i \leq \angle(u_{i-1},u_{i+1},u_i) \leq \angle(u_{i+2},u_{i+1},u_i) \leq \frac{\pi}{2}$.
We may obtain similarly that $\phi_{i+1} \leq \frac{\pi}{2}$.
Hence, the strip bounded by the lines perpendicular to $[u_i,u_{i+1}]$ and passing through its endpoints contains
$H_i \cap H_{i+1}$, and thus, also $\rin(K,o) \B^2$.
\end{proof}


\section{Internal robustness for 3-dimensional convex bodies}\label{sec:internal_3D}

In this section we partly generalize the results of Section~\ref{sec:internal} for convex polyhedra.
Our main result is as follows:

\begin{thm}\label{thm:platonic}
Let $P$ be a regular polyhedron with $S$ faces, $U$ vertices and $H=S+U-2$ edges, and let $o$ be the center of $P$.
Let $P'$ be a convex polyhedron with $S$ faces, $U$ vertices and $H$ edges, each containing an equilibrium point with respect to some $q \in \inter P'$.
Then
\[
\rin(P',q) \leq \rin(P,o),
\]
with equality if, and only if $P'$ is a similar copy of $P$, with $q$ as its center.
\end{thm}

\begin{rem}\label{rem:notall}
Clearly, we do not need to require formally that the numbers of faces, edges and vertices, stable, saddle and unstable points of $P'$ are equal to the corresponding quantities of $P$. We leave it to the reader to show that
if, say, $P'$ has $S$ stable points, $U$ unstable points and $H$ edges, then it has $S$ faces, $U$ vertices and $H$ saddle points.
\end{rem}

In the proof we use the following theorem of Dowker (cf. \cite{FTL}).

\begin{thm}[Dowker]\label{thm:Dowker}
Let $C$ be a unit circle, and $G$ be an $n$-gon circumscribed about (equivalently, containing) $C$.
Then:
\begin{itemize}
\item $\area(G)$ is minimal if, and only if, $G$ is a regular $n$-gon circumscribed about $C$,
\item denoting by $a_n$ the area of a regular $n$-gon circumscribed about $C$, the sequence $\{ a_n \}$
is strictly convex; namely, for any $n -2 > k > 0$, $a_{n-k} + a_{n+k} > 2a_n$. 
\end{itemize}
\end{thm}

\begin{proof}[Proof of Theorem~\ref{thm:platonic}]
For simplicity, assume that $\surf(P') = \surf(P) = 1$, and let $R$ denote the radius of a circle
inscribed in a face of $P$.
Furthermore, for any face $F$ of $P'$, let $r_F$ denote the inradius of $F$, and let $\F$ denote the family of
the faces of $P'$.
Clearly, $\rex(P',q) \leq r = \min \{ r_F : F \in \F \}$.

By Theorem~\ref{thm:Dowker}, we have that $r \leq R$, with equality if, and only if,
all the faces of $P$ contain a stable point, and the faces of $P$ are congruent regular polygons.
We  note that since replacing an $n$-gon by a regular $n$-gon, and replacing a regular $(n-k)$-gon and
a regular $(n+k)$-gon by two regular $n$-gons does not change the total number of edges, we have that,
in the case that $r=R$, the faces of $P'$ are congruent to the faces of $P$ as well.

On the other hand, $\rin(P',q) = r$ implies also that all the lines,
orthogonal to a face and passing through its incenter, meet at $q$.
Furthermore, if the faces $F$ and $F'$ are joined by an edge, then, clearly, the distances of $q$ from the centers of $F$ and $F'$ are equal. Thus, $q$ is at the same distance from any face, which means that the face angles between any two faces joined by an edge are equal. From this, we obtain that the vertex figures of $P$ are congruent, which yields that $P'$ is a regular polyhedron. Hence, the assertion readily follows.
\end{proof}

The following example shows that Theorem~\ref{thm:platonic} is false
without the condition that the numbers of the edges of $P$ and $P'$ are equal.

\begin{example}\label{ex:counterexample}
Let $P$ be a regular tetrahedron of unit surface area with center $o$. Truncate $P$ near a vertex, in such a way that does not change the numbers of the three types of equilibria of $P$, and the truncated part does not intersect the incircle of any face of $P$, and denote the truncated polyhedron by $P'$.
Then $P'$ has the same numbers of stable, saddle and unstable points with respect to any point of $\inter (\rin(P,o) \B)$, but $\surf(P') < \surf (P) = 1$.
Thus, $\rin(P',o) > \rin(P,o)$.
\end{example}

\section{Full robustness, and its variants, in 3 dimensions}\label{sec:full}
In this section, first, we show that the full robustness of equilibrium classes $\{S,U\}$ with $S,U<3$ is maximal, and then we apply our method for some other types of robustness.

\begin{thm}\label{thm:1221}
We have $\rho_{12}=\rho_{21}=\rho_{22}=1$.
\end{thm}

We start by introducing a definition and proving some lemmas.

\begin{defn}
For a convex body $K$ in $\Re^3$, a \emph{bounding box} of $K$ is a brick circumscribed about $K$.
\end{defn}

\begin{lem}\label{lemma:prop}
If for some bounding box of a convex body $K \subset \Re^3$ with edge lengths $a \leq b \leq c$
\begin{itemize}
\item we have $6b \leq c$, then $K$ has at least two unstable points with respect to its center of gravity, and if
\item we have $3a < b$, then $K$ has at least two stable points.
\end{itemize}
\end{lem}

\begin{proof}
Let $B$ be a bounding box of $K$ with edge lengths $a \leq b \leq c$.

First, we show that the distance of the center of gravity of $K$ from any face of $K$ is at least
one quarter of the width of $B$ in that direction. Consider two parallel faces of $B$, say those at the distance $c$ from each other. For simplicity, we assume that the center of gravity of $K$ is the origin $o$,
and these two faces are in the planes $z=z_1 > 0$, and $z=-z_2 < 0$ for some $z_1, z_2 \in \Re$ with $z_1 + z_2 = c$.

Let $K_0$ be the intersection of $K$ with the $(x,y)$-plane and let $p_1$ be a point of $K$ in the plane $z=z_1$.
Let $C$ be the infinite cone with $p$ as its apex and $K_0$ as its base, and let $K'$ be the intersection of $C$
with the infinite strip bounded by the planes $z=z_1$ and $z = -z_2$.
Then, by the definition of the center of gravity of $K$, we have
\[
0 = \int_{(x,y,z) \in K} z \dif x \dif y \dif z \geq \int_{(x,y,z) \in K'} z \dif x \dif y \dif z = z_1 - \frac{3}{4}c.
\]
From this $z_2 \geq \frac{c}{4}$ (and similarly $z_1 \geq \frac{c}{4}$) readily follows.

Now, assume that for the edge lengths of $B$ we have $6a \leq 6b \leq c$.
Observe that for any $q \in \relbd K_0$, we have $|q| \leq \frac{3}{4} \sqrt{a^2+b^2} \leq \frac{\sqrt{2}}{8} c$.
On the other hand, for any point $p \in K$ in $z=z_1$ or in $z=-z_2$, we have
$|p| \geq \frac{c}{4}$.
Since $\frac{c}{4} > \frac{\sqrt{2}}{8} c$, it means that the Euclidean distance function has at least two local maxima,
one with a positive and one with a negative $z$-coordinate.

If $3a < b$, then we may apply a similar argument.
\end{proof}

\begin{lem}\label{lem:lambda}
There is some $\mu > 0$ such that for any $\lambda \in \Re$, there are convex bodies $K_{12}(\lambda) \in \{ 1,2\}$,
$\K_{21}(\lambda) \in \{2,1\}$ and $\K_{22}(\lambda) \in \{2,2\}$, and their bounding boxes $B_{ij}(\lambda)$ with edge lengths $a_{ij}(\lambda) \leq b_{ij}(\lambda) \leq c_{ij}(\lambda)$, where $ij \in \{ 12,21,22 \}$, such that
\begin{itemize}
\item $\lambda  < \frac{c_{12}(\lambda)}{b_{12}(\lambda)}$ and
$\frac{\vol(K_{12}(\lambda))}{\vol(B_{12}(\lambda))} \geq \mu$,
\item $\lambda < \frac{b_{21}(\lambda)}{a_{21}(\lambda)}$ and
$\frac{\vol(K_{21}(\lambda))}{\vol(B_{21}(\lambda))} \geq \mu$, and
\item $\lambda < \frac{c_{22}(\lambda)}{b_{22}(\lambda)}$, $\lambda < \frac{b_{21}(\lambda)}{a_{21}(\lambda)}$
and $\frac{\vol(K_{22}(\lambda))}{\vol(B_{22}(\lambda))} \geq \mu$.
\end{itemize}
\end{lem}

\begin{figure}[!ht]
\begin{center}
\includegraphics[width=0.6\textwidth]{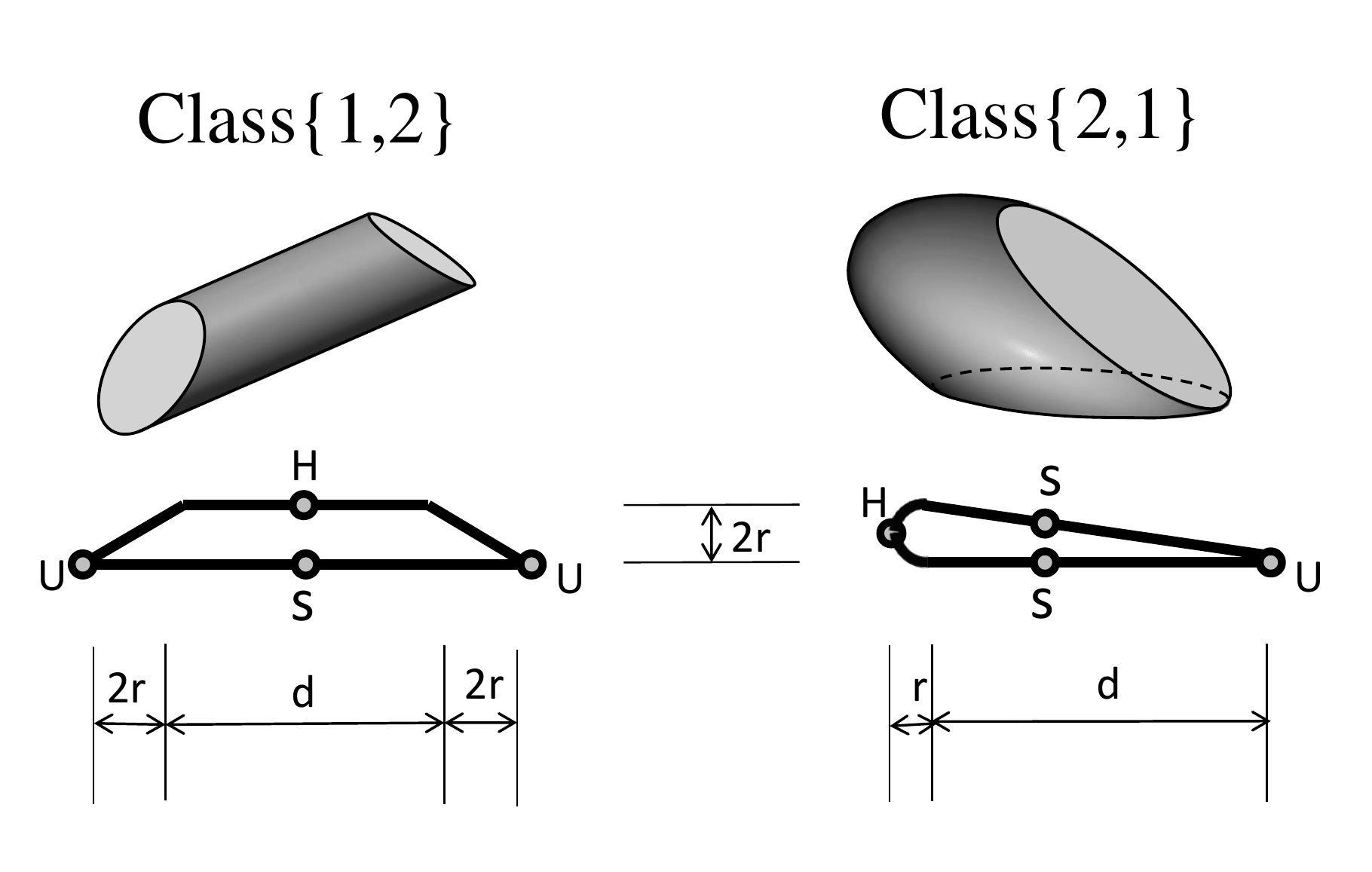}
\end{center}
\caption{Examples of convex bodies in classes $\{1,2\}$ and $\{2,1\}$.}\label{fig:1221}
\end{figure}

\begin{proof}
Figure \ref{fig:1221} shows two convex bodies in $\{1,2\}$ and $\{2,1\}$ with bounding boxes
of edge lengths $a=b=2r$ and $c=d+4r$ and $a=2r$ and $c>b>d$, respectively. The body in $\{1,2\}$ is a suitably truncated
cylinder. We construct the body in $\{2,1\}$ by tilting two, originally coincident circular discs by a small angle. The section
of the two tilted discs with the plane containing their centers are two straight line segments and we connect those
by tangentially attached circular arcs. We remark that the convex hull of a Dupin `needle' Cyclide appears to be in the
same equilibrium class.

Let $\bar{K}_{12}(\lambda)$ and $\bar{K}_{21}(\lambda)$ be these bodies, respectively, with some $d>2r\lambda$.
Then the edge lengths of the given bounding boxes of the bodies satisfy the conditions of the lemma.
To obtain the bodies $K_{12}(\lambda)$ and $K_{21}(\lambda)$ with smooth boundaries, we may apply the standard smoothing algorithm, described, for instance, in \cite{DLS2}.
Then it is an elementary computation to show that the limits of $\frac{\vol(K_{12}(\lambda))}{\vol(B_{12}(\lambda))}$
and $\frac{\vol(K_{21}(\lambda))}{\vol(B_{21}(\lambda))}$ are strictly greater than zero.

As $K_{21}(\lambda)$, we may simply select an ellipsoid with semi-axes $a = 1$, $b=2\lambda $, $c=4 \lambda^2$.
\end{proof}
\noindent Now we return to proving Theorem \ref{thm:1221}.

\begin{proof}
We start with the obvious remark that if any equilibrium class $\{S,U\}$
contains, for every $\varepsilon > 0$, a convex body with downward robustness at least $1-\varepsilon$,
then $\rho_{S,U}=1.$
Hence, using the notation of Lemma~\ref{lem:lambda}, we need only show that
\[
\lim_{\lambda \to \infty} \rho(K_{12}(\lambda)) = \lim_{\lambda \to \infty} \rho(K_{21}(\lambda)) =\lim_{\lambda \to \infty} \rho(K_{22}(\lambda)) = 1.
\]

We start with the class $\{1,2\}$. The only way to reduce the number of
the equilibrium points of $K_{12}(\lambda)$ is to truncate the body in such a way
that the remaining object is in class $\{1,1\}$.
For brevity, we let $K_{11}(\lambda)$ denote a convex body obtained in this way
by a truncation of ``almost'' minimal relative volume (note that a truncation of minimal relative volume may yield
a body with degenerate equilibrium points).
Observe that at least one bounding box of $K_{11}(\lambda)$ fits inside the bounding box of $K_{12}(\lambda)$
described in Lemma~\ref{lem:lambda}.
Thus, by Lemma \ref{lemma:prop} and as $\frac{\vol(K_{12}(\lambda))}{\vol(B_{12}(\lambda))} \geq \mu$
for every value of $\lambda$, we have
\[
\lim_{\lambda \to \infty} \rho(K_{12}(\lambda)) = 1- \lim_{\lambda \to \infty} \frac{\vol(K_{11}(\lambda))}{\vol(K_{12}(\lambda))} = 1.
\]
The argument for the class $\{2,1\}$ runs in an analogous manner.

In case of the class $\{2,2\}$, there are three alternative ways to produce
a convex body with less than six equilibrium points: we may obtain 
a truncated body in one of the classes $\{ 1,2 \}$, $\{2,1 \}$ or $\{ 2,2 \}$.
By Lemma \ref{lemma:prop}, for any bounding box of the truncation with edge lengths $a \leq b \leq c$, 
we have $6b > c$ or $3a \geq b$. Thus, applying an argument similar to the one in the previous cases,
we have $\lim_{\lambda \to \infty}\rho(K_{22}(\lambda)) = 1$.
\end{proof}


Until this point we treated stable and unstable points in a similar manner, robustness corresponded to truncation
resulting in the reduction of the number $N$ of equilibria, regardless of their type. It is natural, and, as we will see, also useful to ask for
\emph{partial robustness}, i.e. the (relative) volume of a truncation necessary to reduce \emph{either S or U}, (the numbers of stable and unstable points, respectively,)
resulting in the notion of $S$-robustness and $U$-robustness, denoted by $\rho^s,\rho^u$, respectively.
Naturally, we have
\[
\rho=\min\{\rho^s,\rho^u\}.
\]
We can define the $S$- and $U$-robustness of equilibrium classes, denoted by $\rho^s_{i,j},\rho^u_{i,j}$, respectively,
in a natural way, and we can immediately see that
\[
\rho^s_{1,n}=\rho^u_{n,1}=1,
\]
because it is not possible to further reduce the number of stable and unstable equilibria in these classes.
Beyond this trivial comment, with ideas very similar to the ones in the proof of Theorem~\ref{thm:1221},
it is easy to obtain one additional result about the partial robustness of two infinite families of equilibrium classes which we formulate in the following theorem.

\begin{thm}\label{thm:bistatic}
If $n>2$ then $\rho^s_{2,n}=\rho^u_{n,2}=1$.
\end{thm}

\begin{proof}
Let $K_{2n}$ and $K_{n2}$ be a regular, planar $n$-gon, and a straight, infinite prism with a regular $n$-gon
as its base.
We leave it to the reader to show for every $\lambda > 0$ the existence of the convex bodies
$K_{2n}(\lambda) \in \{2,n\}$ and $K_{n2}(\lambda) \in \{n,2 \}$ with some bounding boxes satisfying the conditions
in the first two parts of Lemma~\ref{lem:lambda}.
These convex bodies approach $K_{2n}$ and $K_{n2}$, respectively, as $\lambda \to \infty$.
Then, to prove the assertion, we may apply Lemma~\ref{lemma:prop} and follow the idea of Theorem~\ref{thm:1221}.
\end{proof}

Theorem~\ref{thm:bistatic} leads to the following Corollary.

\begin{cor} \label{cor:bistatic}
If $n>2$ then
$\rho_{2,n}=\rho^{u}_{2,n}$
$\rho_{n,2}=\rho^{s}_{n,2}.$
\end{cor}

\section{Remarks and open problems}\label{sec:summary}
In this paper we investigated the robustness of a convex solid with $N$ equilibrium points. 
Beyond the original definition of `full' robustness we also defined internal and external, in the latter case also upward and downward robustness. 

\subsection{Smoothness and structural stability}

Before formulating conjectures and raising questions we make some remarks about smoothness and structural stability.

\begin{rem}
We have examined convex bodies with piecewise $C^\infty$-class boundaries.
Nevertheless, all our results (and proofs) can be applied also for bodies with piecewise $C^2$-class boundaries.
\end{rem}

\begin{rem}\label{rem:ovalstability}
Let $S \geq 3$. For every $\varepsilon > 0$, there is a $\delta=\delta(\varepsilon,S) > 0$,
with $\lim_{\varepsilon \to 0+0} \delta = 0$, such that
if $K \in \K_2$ has $n$ stable points with respect to $p \in \inter K$, and $\rin(K,p)> \frac{1}{2S} - \varepsilon$,
then the Hausdorff distance of $K$ and a regular $S$-gon, with $p$ as its centre, is less than $\delta$.
\end{rem}

\begin{rem}
Remark~\ref{rem:ovalstability} holds also if we replace $\rin(K,p)$ with $\rex(K,p)$.
\end{rem}

\begin{rem}\label{rem:ovalsaregood}
For every $S \geq 3$, the maximum of $\rin(K,p)$ over $\K_2$ can be approached by regions with smooth boundaries as well; or in other words, $\frac{1}{2S}=\sup \{ \rin(K,p): K \in \O_2, p \in \inter K \}$. To show this, it suffices
to replace the boundary of a regular $S$-gon near the vertices by suitable elliptic arcs. This provides a $C^2$-class curve as the boundary of a convex region, which, after applying a smoothing algorithm like in \cite{DLS2},
yields a convex region $K \in \O_2$, with internal robustness ``almost'' equal to $\frac{1}{2S}$.
\end{rem}

\subsection{Conjectures and open questions}
Since we know  \cite{DomokosRuina} that the first nonempty class in the plane is  $\{2\}$,  trivially we have
\[
\rho_2=1.
\]
Based on Theorems~\ref{thm:external} and \ref{thm:internal}, we can formulate the next conjecture.
\begin{conj}\label{2D}
$\rho_n=\rho(\mbox{regular n-gon})$ if  $n>2$.
\end{conj}

In 3 dimensions the situation is less transparent, however, there are many interesting questions.
As a modest generalization of Theorem~\ref{thm:platonic} we propose the following.

\begin{conj}\label{conj:general3D}
Theorem \ref{thm:platonic} is also valid for downward external robustness, and also for downward full robustness.
\end{conj}

A bolder~generalization of Theorem \ref{thm:platonic} refers to equilibrium classes.
We call the equilibrium classes containing platonic solids (classes $\{4,4\}, \{6,8\}, \{8,6\}, \{20,12\}$ and $\{12,20\}$) briefly  \em platonic classes\rm. An affirmative answer for Conjecture~\ref{conj:general3D}
would suggest that  in the platonic classes platonic solids have maximal downward full robustness.
This we pose as an open question:

\begin{prob}
Prove or disprove that in the platonic classes platonic solids have maximal downward full robustness.
\end{prob}

So far we addressed the robustness of nine classes in $\Re^3$: in Section~\ref{sec:full} we discussed the four with indices less than 3
(Theorem~\ref{thm:1221}), and here we formulated a conjecture related to the five platonic ones.
In addition, Theorem~\ref{thm:bistatic} and Corollary~\ref{cor:bistatic}
we addressed the $S$- and $U$-robustness of two infinite class families; here we complement the latter results with

\begin{conj}
$\rho_{2,n}=\rho_{n,2}=\rho_{n}$.
\end{conj}

The underlying geometric idea is that in classes $\{2,n\}$ and $\{n,2\}$ the very same solids 
(described in the proof of Theorem~\ref{thm:bistatic}) which have
`maximal' $S$- and $U$-robustness, also have `maximal' overall robustness.
In general, one would expect that a `uniform distribution' of equilibrium points is associated with maximal robustness. If all equilibrium points were of the same type, this would resemble the Tammes Problem
\cite{Tammes}, however, since we have different types of equilibria, this analogy is incomplete.
While our understanding of the robustness of general equilibrium classes is rather limited,
the following conjecture is strongly suggested by Corollary~\ref{cor:monotonicity} and is consistent
with all our previous findings and conjectures:

\begin{conj}
If $i\geq k$ and $j\geq l$ then $\rho_{i,j}\leq\rho_{k,l}$.
\end{conj}
In general, it is an interesting question whether $\rho_{i,j}=\rho_{j,i}$.
From the point of view of applications and numerical experiments it might be of interest to study \emph{partial robustness}, measuring the difficulty of reducing 
separately either $S$ or $U$.
For example, based on the above arguments and conjectures, we expect in classes $\{2,n\}$ the partial $S$-robustness to be much smaller then the partial $U$-robustness
and the inverse statement applies for classes $\{n,2\}$.


Another `reduced', however potentially important version of robustness
can be defined by admitting only truncations by planes (or straight
lines in 2 dimensions). This case is not only interesting because
numerical experiments appear feasible but also this is the only
truncation of a convex body where both resulting objects remain
(weakly) convex. This naturally leads to the statistical study of the
evolution of numbers of equilibria in a population of convex solids
which are generated by subsequent planar truncations and where
\emph{both} convex solids resulting from a trucation are considered.
 
This idea also suggests an alternative definition of robustness. So
far we considered minimal truncations changing the number of
equilibria, nevertheless it is also possible to regard the average
magnitude of such truncations. While it is not clear how to define a
measure directly on the space of general convex truncations, we can
obtain a possible approximation by admitting successive truncations by (hyper)planes.
The space of the latter is a special type affine Grassmannian manifold
with a natural measure $\mu$, invariant under the motions of the
embedding Euclidean space (cf. \cite{Santalo}), and can be extended in
a natural way to a measure $\mu_n$ on the space of $n$ successive
planar truncations. We call a truncation nontrivial if it intersects
$K$ and we call it neutral, if it leaves the equilibrium class
invariant. Then we can define the \emph{n-th order average robustness}
$\rho_n(K)$ associated with a convex solid $K$ as the measure of the
set of $n$, subsequent neutral truncations divided by the measure of
the set of $n$, subsequent nontrivial truncations. Such a concept may be useful in
computations and practical applications, also, it would be of interest
to see whether solids with maximal average robustness also have
maximal (downward) robustness in the original sense.
 
One possible, practical choice for visualizing
average robustness is to assume the above-mentioned
natural measure on the space of truncations and then plot the
percentage of truncations resulting in a given change $\Delta N$ of the
number of equilibrium points versus the relative volume corresponding
to the truncation. This visualization admits the representation
of an arbitrary convex solid of arbitrary dimension on the unit square.
The area on the unit square corresponding to neutral truncations
($\Delta N = 0$) could also serve as a definition of average robustness.
In Figure \ref{fig:new} we show this plot for the square: the lower, middle and upper regions
represent the truncations resulting in a polygon with $S=3$, $S=4$ and $S=5$ stable points,
corresponding to $\Delta N=-1,0,1$, respectively.

\begin{figure}[here]
\includegraphics[width=0.9\textwidth]{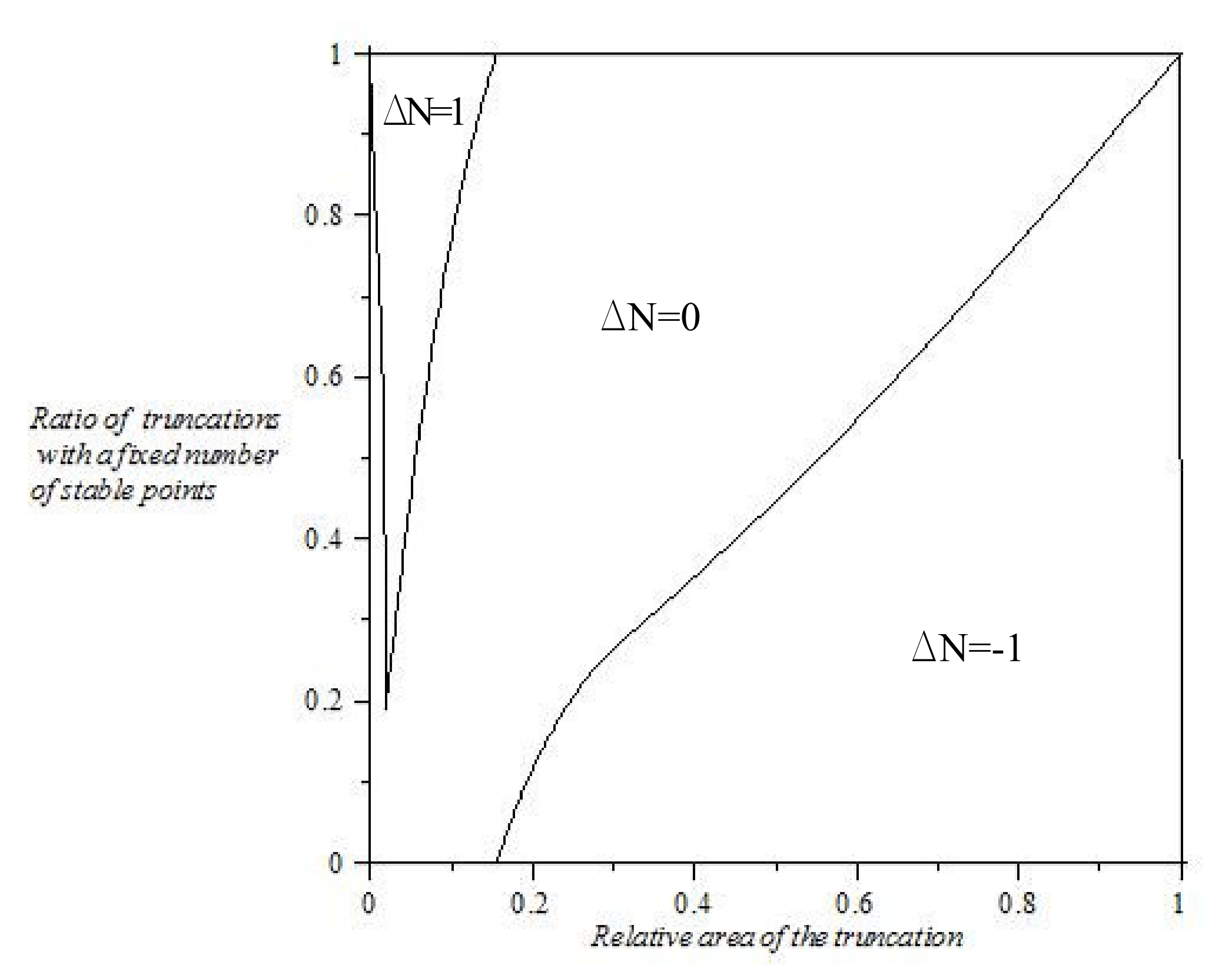}
\caption[]{Truncations of a unit square with one line.}
\label{fig:new}
\end{figure}


As pointed out in the introduction, our results may help to understand natural abrasion processes.
Theorems~\ref{thm:1221} and \ref{thm:bistatic} indicate that if we study the statistical outcome of a natural abrasion process,
the fraction of solids in classes $\{1,n\}, \{n,1\}$ would be rather low as transition from neighbour classes $\{2,n\}, \{n,2\}$ is prohibited
by the maximal $S$- and $U$-robustness of the latter. Indeed, field studies \cite{Varkonyi, DSSV} appear to confirm this: in coastal reagions
the percentage of pebbles in classes $\{1,n\}, \{n,1\}$ was found to be below 0.1\%.

\section{Acknowledgement}
The supports of the Hungarian Research Fund (OTKA) grant 104601, and of the J\'anos Bolyai Research Scholarship of the Hungarian Academy of Sciences, is gratefully acknowledged.
The authors thank M\'ark Mezei for his valuable help in computing Figure \ref{fig:new}.


\begin{thebibliography}{99}
\bibitem{Archimedes} T. I. Heath (ed.) \emph{The Works of Archimedes}, Cambridge University Press, 1897.
  
\bibitem{Arnold} Arnold, V.I., 
\emph{Ordinary differential equations}, 10th printing, MIT Press, Cambridge, 1998.

\bibitem{Bloore} Bloore, F.J.,
\emph{The Shape of Pebbles}, Math. Geol. {\bf 9} (1977) 113-122.
	
\bibitem{Conway} Conway, J.H. and Guy, R.,
\emph{Stability of polyhedra}, SIAM Rev. {\bf 11} (1969), 78-82.
	
\bibitem{Dawson} Dawson, R.,
\emph{Monostatic Simplexes}, Amer. Math. Monthly {\bf 92} (1985), 541-546.
	
\bibitem{DawsonFinbow} Dawson, R. and Finbow, W.,
\emph{What shape is a loaded die?}, Math. Intelligencer {\bf 22} (1999), 32-37.

\bibitem{DoCarmo} Do Carmo M. P., \emph{Differential Geometry of Curves and Surfaces},
Prentice-Hall Inc., Englewood Cliffs, New Jersey, 1976.
	
\bibitem{DomokosGibbons} Domokos G., Gibbons G.W., \emph{The evolution
of pebble shape in space and time}. Proc. Roy. Soc. London (2012), DOI:10.1098/rspa.2011.0562.

\bibitem{DLS} Domokos G., L\'angi Z. and Szab\'o, T.,
\emph{On the equilibria of finely discretized curves and surfaces}, Monatsh. Math. (2012),
DOI: 10.1007/s00605-011-0361-x.

\bibitem{DLS2} Domokos G., L\'angi Z. and Szab\'o, T.,
\emph{The genealogy of convex solids}, submitted, arXiv:1204.5494.

\bibitem{DomokosRuina} Domokos G., Ruina A. and Papadopoulos, J.,
\emph{Static equilibria of rigid bodies: is there anything new}, J. Elasticity 
{\bf 36} (1) (1994) 59-66

\bibitem{DSSV} Domokos G., Sipos A.\'A., Szab\'o, T.,V\'arkonyi P., \emph{Pebbles, shapes and equilibria},
Math. Geosci. {\bf 42}(1) (2010), 29-47.

\bibitem{DSV} Domokos G., Sipos A.\'A. and  V\'arkonyi P.
\emph{Continuous and discrete models for abrasion processes}, Per. Pol. Arch. {\bf 40} (2009), 3-8.
 
\bibitem{DV} Domokos G. , V\'arkonyi P., \emph{Geometry and self-righting of turtles}, Proc. Roy.Soc. London B.
{\bf 275}(1630) (2008), 11-17.

\bibitem{FTL}, Fejes T\'oth, L., \emph{Regular Figures}, Pergamon, Oxford, 1964.

\bibitem{Heppes} Heppes, A.,
\emph{A double-tipping tetrahedron}, SIAM Rev. {\bf 9} (1967), 599-600.

\bibitem{Krapivsky} Krapivsky, P.L. and Redner S.,
\emph{Smoothing a rock by chipping}, Phys. Rev. E {\bf 9} (2007), 75(3 Pt 1):031119.

\bibitem{Milnor} Milnor, J., \emph{Morse Theory}, Princeton Univ. Press, New Jersey, 1963.

\bibitem{Poston} Poston, T. and Stewart, J.,
\emph{Catastrophe theory and its applications}, Pitman, London, 1978.	

\bibitem{Santalo} Santal\'o, L. A., \emph{Integral Geometry and Geometric Probability},
Addison-Wesley, Reading, MA, 1976. 

\bibitem{SOA} Simsek, A., Ozdaglar, A. and and Acemoglu, D.,\emph{Generalized Poincar\'e-Hopf Theorem for compact nonsmooth regions}, Math. Oper. Res. {\bf 32}(1), 193–214.

\bibitem{SDWH} Sipos A.\'A., Domokos G., Wilson A. and Hovius N. 
\emph{A Discrete Random Model Describing Bedrock Erosion}, Math. Geosci. {\bf 43} (2011), 583-591.

\bibitem{Tammes} Tammes P.M.L. \emph{On the origin of number and arrangement of the places of exit on pollen grains}. Diss. Groningen (1930).


\bibitem{Varkonyi} V\'arkonyi, P.L. and Domokos, G.,
\emph{Static equilibria of rigid bodies: dice, pebbles and the Poincar\'e-Hopf Theorem},
J. Nonlinear Sci. {\bf 16} (2006), 255-281.

		
\bibitem{Zamfirescu} Zamfirescu, T., 
\emph{How do convex bodies sit?}, Mathematica  {\bf 42} (1995), 179-181.

\end{thebibliography}
\end{document}